\documentclass[journal]{IEEEtran}
\usepackage{cite}
\usepackage{caption}
\usepackage{subfigure}
\usepackage[caption=false]{subfig}
\usepackage{scalerel}
\usepackage[utf8]{inputenc} 
\usepackage{amsmath,color,amsthm,amssymb,graphicx,amsfonts,float}
\newtheorem{lem}{Lemma}

\usepackage{graphicx} 
\usepackage{amsfonts,amsmath} 

\usepackage{booktabs} 
\usepackage{array} 
\usepackage{paralist} 
\usepackage{verbatim} 
\usepackage{subfig} 
\usepackage{multirow}
\usepackage{comment}
\usepackage{algorithm,algorithmic}
\usepackage[dvipsnames]{xcolor}
\usepackage{svg}

\newcommand{\yajing}[1]{{\color{red} #1}}

\newcommand{\april}[1]{{\color{ForestGreen} #1}}
\newcommand*{\Scale}[2][4]{\scalebox{#1}{$#2$}}%

\title{Decentralized Low-Rank State Estimation for Power Distribution Systems}
\author{April Sagan, Yajing Liu, Andrey Bernstein 
\thanks{A. Sagan is with  the Department of  Mathematical Sciences, Rensselaer Polytechnic Institute, Troy, NY 12180, {\small sagana@rpi.edu}, and is supported by the National Science Foundation under Grant Number DMS-1736326.}
\thanks{Y. Liu and A. Bernstein are with Power System Engineering Center, National Renewable Energy Laboratory, Golden, CO 80401, {\small \{yajing.liu, andrey.bernstein\}@nrel.gov.}}
}

\newtheorem{thm}{Theorem}

\begin{document}
\maketitle



\begin{abstract}
This paper considers the low-observability state estimation problem in power distribution networks and develops a decentralized state estimation algorithm leveraging the matrix completion methodology. Matrix completion has been shown to be an effective technique in state estimation that exploits the low dimensionality of the power system measurements to recover missing information. 
This technique can utilize an approximate (linear) load flow model, or it can be used with no physical models in a network where no information about the topology or line admittance is available.

The direct application of matrix completion algorithms requires solving a semi-definite programming (SDP) problem, which becomes infeasible for large networks. We therefore develop a  decentralized algorithm that capitalizes on the popular proximal alternating direction method of multipliers (proximal ADMM). The method allows us to distribute the computation among different areas of the network, thus leading to a scalable algorithm. By doing all computations at individual control areas and only communicating with neighboring areas, the algorithm eliminates the need for data to be sent to a central processing unit and thus increases efficiency and contributes to the goal of autonomous control of distribution networks. We illustrate the  advantages of the proposed algorithm numerically using standard IEEE test cases.
\end{abstract}

\section{Introduction}
The increased penetration of distributed energy resources (DERs) (such as solar photovoltaics, wind turbines, and energy storage systems) in power distribution systems introduces bi-directional power flows that impact the system responses to various types of disturbances. As a result, distribution system state estimation, the procedure of obtaining the voltage phasors of all the system buses at a given point in time, has become of central importance  to   maintain the normal and secure operating conditions. Unlike the transmission systems, the limited availability of real-time measurements from Supervisory Control and Data Acquisition (SCADA) systems as well as from Phasor Measurement Units (PMUs) renders the well-developed weighted least squares (WLS) method \cite{Abur2004} inapplicable. This is true due to the \emph{low-observability} condition, wherein the estimation problem is \emph{underdetermined}.

To address the low observability challenge, there are two standard approaches in the literature. In the first approach, sensor placement strategies are devised to derive additional measurements required for full observability \cite{singh2009, bhela2018, YCNN}. In the second approach, machine learning and signal processing tools are used to derive \emph{pseudo-measurements} using existing sensors and/or historical data \cite{Clements2011,manitsas2012, wu2013,zhou2019gradient}. Considering the high expenses of installing sensors and large error introduced by pseudo-measurements \cite{Clements2011},  there are recently several attempts in the literature to directly address the underdetermined estimation problem by leveraging the \emph{low-rank structure} of the measurements \cite{gao_wang2016,Genes2019,Liao2019,Schmitt}. These approaches can be viewed as \emph{regularized} WLS methods, wherein the regularization term is aimed at minimizing the rank of the data matrix. In particular, \cite{Schmitt} proposed a  distribution system state estimation method by augmenting the traditional matrix completion \cite{exactMatrixCompletion} with noise-resilient power flow constraints, and demonstrated that this constrained matrix completion method works well under low-observability conditions. In \cite{gao_wang2016, Liao2019}, the standard matrix completion method was applied to recover missing PMU data  and \cite{Genes2019} applied the standard matrix completion method with Bayesian estimation to recover missing low voltage distribution systems data.


In this paper, we start with the formulation of \cite{Schmitt}, in which the data matrix was built using a \emph{single snapshot} of data. Each row in the matrix represents a power system bus and each column represents a quantity relevant to that bus (e.g., voltage and power injection).
We extend this formulation to a \emph{multi-period} problem, leveraging the spatial and temporal correlation in the time-series PMU data. We then develop a  decentralized algorithm that capitalizes on the popular proximal alternating direction method of multipliers (proximal ADMM) \cite{proxADMM}. The method allows to distribute the computation among different areas of the network, thus leading to a scalable algorithm. By doing all computations at individual control areas and only communicating with neighboring areas, the algorithm eliminates the need for data to be sent to a central processing unit and thus increases efficiency and contributes to the goal of autonomous control of distribution networks.  We show algorithm's efficiency numerically using standard IEEE test systems.

The main contributions of this paper are as follows:
\begin{itemize}
    \item We develop a scalable decentralized matrix completion algorithm that avoids solving the computationally expensive semi-definite programming (SDP) problem, which becomes infeasible for large networks.
    \item Our method is able to leverage both spatial and temporal correlation simultaneously, in contrast to the existing literature.
    \item We show that the computation of the linear approximations to the load-flow equations can be efficiently decentralized due to the approximate block structure of the coefficient matrix. 
\end{itemize}

The rest of the paper is organized as follows. Section \ref{section:ProblemFormulate} poses the state estimation problem as a regularized matrix completion problem.  An approximation of the linear load flow model which allows for decentralized computations is derived  in Section \ref{Section:DecentralziedLinearPowerFlow}. Section~\ref{section:decentralized}  develops a decentralized proximal ADMM algorithm to solve the posed problem.
Numerical results are presented in Section \ref{section:numericalResults} and the paper is concluded in Section \ref{section:conclusions}.

\section{Problem Formulation}
\label{section:ProblemFormulate}
For the scope of this paper, 
we consider a general three-phase distribution network consisting of one slack bus and a given number of multi-phase $PQ$ buses. 
\subsection{Matrix Completion for State Estimation}
\label{subsection:ProblemFormulate_a}
Given a matrix known to be low rank (or approximately low rank) where {a small subset of entries are observed}, the low-dimensionality of the matrix can be exploited to reconstruct the missing entries. {Let $\Scale[0.9]{M\in\mathbb{R}^{m\times n}}$ ($m<n$) be the matrix that we wish to reconstruct, and let $\Scale[0.9]{\Omega\subset \{1, \ldots, m\} \times \{1, \ldots, n\}}$ denote the set of known elements of $\Scale[0.9]{M}$. Define the observation operator $\Scale[0.9]{{P}_\Omega:\mathbb{R}^{m\times n} \rightarrow \mathbb{R}^{m\times n}} $ as: $\Scale[0.9]{{P}_\Omega(M)_{ij}=
	M_{ij}}$ if  $(i,j) \in \Omega$; 0 otherwise.

 The goal is to recover $\Scale[0.9]{M}$ from $\Scale[0.9]{{P}_\Omega(M)}$, when the number of observation is much smaller than the number of entries in $\Scale[0.9]{M}$. Under the assumption that $\Scale[0.9]{M}$ is low rank, the problem is formulated as \cite{exactMatrixCompletion}}
\[
\Scale[0.9]{
\begin{aligned}
& \underset{X\in\mathbb{R}^{m\times n}}{\text{minimize}}
& & \mathrm{rank}(X) +\frac{\mu}{2} || {P}_\Omega(X)-{P}_\Omega(M)||_F^2,
\end{aligned}}
\]
{where $\Scale[0.9]{X}$ is the decision variable, rank($\Scale[0.9]{X}$) is the rank of  $\Scale[0.9]{X}$, and $\Scale[0.9]{\mu > 0}$ is a weight parameter. }

{Unfortunately, the above formulation is of no practical use because the rank minimization is NP-hard. The most common way to solve this problem is to minimize the nuclear norm instead, which is the tightest  convex relaxation of the rank function. Denote the nuclear norm of $\Scale[0.9]{X}$ by} 
$\Scale[0.9]{||X||_*=\sum_{i=1}^{m} \sigma_i(X)}$,
where $\Scale[0.9]{\sigma_i(X)}$ is the $i${th} {largest} singular value of $X$. 
The {heuristic optimization} problem is then formulated as 
\begin{equation} 
\Scale[0.9]{
\begin{aligned}
\label{relaxedMC}
& \underset{X\in\mathbb{R}^{m\times n}}{\text{minimize}}
& & ||X||_* +\frac{\mu}{2}|| {P}_\Omega(X)-{P}_\Omega(M)||_F^2.
\end{aligned}}
\end{equation}

{The authors of \cite{exactMatrixCompletion} proved that if $\Scale[0.9]{\Omega}$ is sampled uniformly at random from $\Scale[0.9]{M}$ and no noise is present, then with high probability, the solution to the convex optimization problem of minimizing the nuclear norm with the constraint $\Scale[0.9]{P_\Omega(X)=P_\Omega(M)}$ is exactly $\Scale[0.9]{M}$, provided that the number of known entries is larger than $\Scale[0.9]{C n^{1.2}\cdot \text{rank}(X)\cdot \log(n)}$ for some positive constant $\Scale[0.9]{C}$.  Additionally, the solution to \eqref{relaxedMC} has been shown to be sufficiently close to the original matrix $\Scale[0.9]{M}$ even when noise is present \cite{Chi2019}.}

Matrix completion has been shown to be effective in state estimation of power systems by \cite{Schmitt}, which formulated the data matrix $\Scale[0.9]{M}$ using single-period information. In this paper, we extend the formulation of \cite{Schmitt} to a multi-period problem, wherein  we set up the data matrix $\Scale[0.9]{M}$ by including power system measurements for a  consecutive time series.

 Assume that the voltage phasor and other measurements at the slack bus are known. Thus, we will only use voltage phasors and measurements at nonslack (PQ) buses to form the data matrix. {Let $\Scale[0.9]{\mathcal{P}}$ denote the set of phases at all nonslack buses and $\Scale[0.9]{|\mathcal{P}|}$ be the corresponding cardinality.} 
 The measurements we will use in this matrix are the real voltage, imaginary voltage, voltage magnitude, active and reactive power injection at each phase of the nonslack buses. Consider a time series $\Scale[0.9]{t=1,\ldots, T}$. Let $\Scale[0.9]{M^t}$ denote the measurement matrix at time $t$ {such that each column represents a phase and each row represents a quantity relevant to the phase. To be specific, for each phase $\Scale[0.9]{i\in\mathcal{P}}$,
the corresponding column of $\Scale[0.9]{M^t}$ is of the form}
\[\Scale[0.9]{\begin{aligned}
\left[\Re(v_i),\ \Im(v_i),\ |v_i|,\ \Re(s_i), \ \Im(s_i)\right]^{\intercal},
\end{aligned}}\]
where  $\Scale[0.9]{\Re(\cdot)}$ and $\Scale[0.9]{\Im(\cdot)}$ are the real part and imaginary part of a complex variable, respectively, $\Scale[0.9]{{\intercal}}$ is the transpose notation, $\Scale[0.9]{v=[v_1,\ldots, v_{|\mathcal{P}|}]^{\intercal}\in\mathbb{C}^{|\mathcal{P}|}}$ is the vector containing voltage phasors at each phase of nonslack buses, $\Scale[0.9]{s=[s_1,\ldots, s_{|\mathcal{P}|}]^{\intercal}\in\mathbb{C}^{|\mathcal{P}|}}$ is the vector of power injections at each phase of nonslack buses.
{Then the matrix $\Scale[0.9]{M}$ is constructed by 
\begin{equation}
\label{eqn:dataMatrix}
\Scale[0.9]{
\begin{aligned}
M=\left[M^1;\ M^2;\ \cdots; M^{T}\right] \in \mathbb{R}^{m \times n},
\end{aligned}}
\end{equation}
where  $\Scale[0.9]{T}$ is the  number of time steps and $\Scale[0.9]{m = 5T, n=|\mathcal{P}|}$.}

\begin{figure}
   \centering
   \includegraphics[scale=0.38,trim= 50 00 45 0, clip]{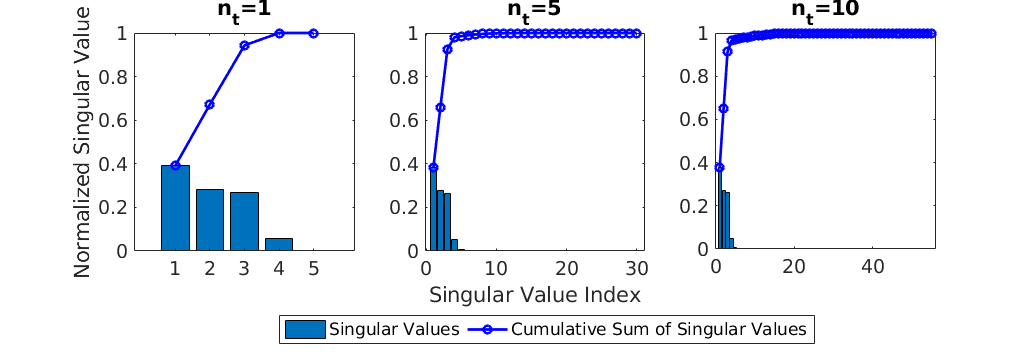}
   \caption{Singular value distribution of fully observed measurement matrix formulated as in \eqref{eqn:dataMatrix} for 1, 5, and 10 time steps.}
   \label{fig:SingularValues}
\end{figure}
As their entries are all related, such a matrix has low dimensionality.  Fig.~\ref{fig:SingularValues} demonstrates the approximate low rank property of matrix $\Scale[0.9]{M}$ for the IEEE 123-bus feeder using 1, 5, and 10 time-step data, respectively. With more time steps, the rank increases slightly, however, due to the increased amount of rows, the low dimensionality of the matrix can be more easily exploited. Additionally, the approximate low-rank property of time-series data containing voltage and current phasors from PMUs has been shown in \cite{pmuPaperEvaluationofAlgs}. 

Though we limit our analysis to just these measurements, we remark that this model is able to accommodate any measurements known to have some correlation with other measurements in the matrix.


\subsection{Linear Load-Flow Model} \label{linearLoadFlow}

{Classical matrix completion methods may output the reconstructed matrix that
differs significantly from the true underlying matrix even if the data satisfies the
low-rank assumption. This is because  the information contained in the observed data is insufficient. To address the problem, we include the power flow constraints in the formulation as in \cite{Schmitt}.}

{In order to avoid the nonlinearity of the voltage-load model,}
we leverage a linear power flow model of the form

\begin{equation}
\Scale[0.9]{
\label{eqn:linvol}
   \begin{aligned}
{v}&\approx w+Nh,\\
|{v}|&\approx|w|+K{h},
\end{aligned}}
\end{equation}
 where $\Scale[0.85]{h=[\Re(s)^{\intercal} \quad \Im(s)^{\intercal}]^{\intercal}\in\mathbb{R}^{2|\mathcal{P}|}.}$
See, for example, \cite{linearLoadFlow,linearLoadFlow2,Christakou2013,Giannakis2016}  on how to compute the model coefficients $\Scale[0.9]{N,K\in\mathbb{C}^{|\mathcal{P}|\times 2|\mathcal{P}|}}$, and $\Scale[0.9]{w\in\mathbb{C}^{|\mathcal{P}|}}$.
 
 Letting $\Scale[0.9]{v^t,w^t,h^t,N^t,K^t}$  denote the corresponding $\Scale[0.9]{v, w,  h},$ $\Scale[0.9]{N, K}$ in (\ref{eqn:linvol}) at time slot $t$, we have the linear approximation of voltage phasors and magnitudes at time $t$.
Let $\Scale[0.85]{A^t={[{(N^t)}^{\intercal}\ {(K^t)}^{\intercal}]}^{\intercal}}$. For simplicity of expression in the sequel, we use the following model to express the above linear model at times $\Scale[0.9]{t=1,\ldots, T}$,
 \begin{equation}
\label{eqn:linmodel}
y\approx Ax+b,
\end{equation}
where \[\Scale[0.9]{y = \begin{bmatrix}
    v^1\\
    |v^1|\\
    \vdots\\
    v^{T}\\
    |v^{T}|
\end{bmatrix}, A = \begin{bmatrix}
    A^1 &  &\\
    &\ddots &\\
     & & A^{T}
\end{bmatrix},  x= \begin{bmatrix}
    h^1\\
    \vdots\\
    h^{T}
\end{bmatrix}, \ \text{and} \
 b =\begin{bmatrix}
    w^1\\
    |w^1|\\
    \vdots\\
    w^{T}\\
    |w^{T}|
\end{bmatrix}}.
\]
{The linear model (\ref{eqn:linmodel}) will be incorporated into the matrix completion formulation in the following section.}
\subsection{The Full Model}
By adding the linear constraint (\ref{eqn:linmodel}) as  a regularization term in the objective function of (\ref{relaxedMC}), we obtain the following matrix completion problem:
\begin{equation}
\begin{aligned}
\label{centralized_SE}
&{\text{minimize}}
\quad \Scale[0.87]{||X||_*+\frac{\mu}{2}||{P}_\Omega(X- M)||_F^2+\frac{\nu}{2}||y-(Ax+b)||_2^2},\\
&\text{over} \quad\quad\quad \Scale[0.87]{X\in\mathbb{R}^{m\times n},y\in\mathbb{C}^{\frac{2}{5}mn},
 x\in\mathbb{R}^{\frac{2}{5}mn}},\\
&\text{s.t.} \quad\quad\quad\ \ \Scale[0.87]{y = \left[
    a_1^{\intercal}X \ \, 
    a_2^{\intercal}X \, 
    \ldots \,
    a_{_{2T-1}}^{\intercal}X\ \,
    a_{_{2T}}^{\intercal}X 
\right]^{\intercal}}, \\
&  \quad\quad\quad\ \ \ \ \ \Scale[0.87]{x = \left[
   c_1^{\intercal}X\ \,
     c_2^{\intercal}X \,
    \ldots \,
    c_{_{2T-1}}^{\intercal}X\ \,
    c_{_{2T}}^{\intercal}X
         \right]^{\intercal}},
\end{aligned}
\end{equation}
where  $\Scale[0.9]{\mu,\nu>0}$ are  penalty parameters, and the parameters $\Scale[0.9]{a_{_1},\ldots, a_{_{2T}}, c_{_{1}},\ldots, c_{_{2T}}\in\mathbb{C}^{m}}$ have the following specific forms: For $\Scale[0.85]{1\leq t\leq T}$ and $\Scale[0.9]{i=1,2, a_{_{2(t-1)+1}}=e_{_{5(t-1)+1}}+je_{_{5(t-1)+2}}}$,  $\Scale[0.9]{a_{_{2t}}=e_{_{5(t-1)+3}}}$, $\Scale[0.9]{c_{_{2(t-1)+i}}=e_{_{5(t-1)+3+i}}}$, in which $t$ denoting the $t$-th time step, $i$ denoting the $i$-th $c$, and $\Scale[0.9]{e_{_k} \ (1\leq k\leq m)}$ is the basis vector in $\Scale[0.9]{\mathbb{R}^m}$. For example, for $a_1$, using the first expression, $\Scale[0.9]{t=1,a_{_1}=e_{_1}+je_{_2}}$; for $a_2$, using the second expression, $\Scale[0.9]{t=1,a_{_2}=e_{_3}}$; for $c_1$, using the third expression, $\Scale[0.9]{t=1,i=1, c_{_1}=e_{_4}}$; for $c_2$, using the third expression,  $\Scale[0.9]{t=1, i=2, c_{_2}=e_{_5}}$.

In the next sections, our goal is to develop a decentralized algorithm to solve \eqref{centralized_SE}.

 \section{Decentralized Linear Load-Flow Model}
\label{Section:DecentralziedLinearPowerFlow}
In order to develop a decentralized algorithm to solve \eqref{centralized_SE}, we first must find a decentralized approximation to the linear load-flow model of Section \ref{linearLoadFlow}. {Assume that the distribution system is partitioned into $\Scale[0.9]{n_A}$ areas.} 
{By linear model (\ref{eqn:linvol}), at time $t$, for a phase $i$ in area $\Scale[0.8]{l}$ {($\Scale[0.8]{1\leq l\leq n_A}$)},  we have 
\[\begin{bmatrix}
         v_i^t\\
         \ \\
         |v_i^t|
     \end{bmatrix} \approx \Scale[0.85]{\begin{bmatrix}
        \sum\limits_{j=1}^{2|\mathcal{P}|}N_{ij}^th_j^t\\
        \ \\ 
         \sum\limits_{j=1}^{2|\mathcal{P}|}K_{ij}^th_j^t
     \end{bmatrix}}+\begin{bmatrix}
          w_i^t\\
          \ \\ 
         |w_i^t|
     \end{bmatrix},
\]
in which 
\[\Scale[0.8]{\begin{bmatrix}
        \sum\limits_{j=1}^{2|\mathcal{P}|}N_{ij}^th_j^t\\
        \ \\ 
         \sum\limits_{j=1}^{2|\mathcal{P}|}K_{ij}^th_j^t
     \end{bmatrix}= \begin{bmatrix}
        \sum\limits_{j\in A^l}P_{ij}^t+\sum\limits_{k\in\mathcal{N}(l)}\sum\limits_{j\in A^k}P_{ij}^t+\sum\limits_{k\notin\mathcal{N}(l)}\sum\limits_{j\in A^k}P_{ij}^t\\
        \ \\ 
         \sum\limits_{j\in A^l}Q_{ij}^t+\sum\limits_{k\in\mathcal{N}(l)}\sum\limits_{j\in A^k}Q_{ij}^t+\sum\limits_{k\notin\mathcal{N}(l)}\sum\limits_{j\in A^k}Q_{ij}^t
     \end{bmatrix}}
\]
with $\Scale[0.85]{P_{ij}^t=N_{ij}^th_j^t, Q_{ij}^t=K_{ij}^th_j^t}$, $\Scale[0.85]{{A}^l}$ denoting the set of phases in area $l$, and $\Scale[0.85]{\mathcal{N}(l)}$ denoting the areas adjacent to area $l$.
We note that the terms $\Scale[0.85]{\sum_{k \notin \mathcal{N}(l) }\sum_{j\in  A^k}P_{ij}^t}$ and $\Scale[0.85]{\sum_{k \notin \mathcal{N}(l) }\sum_{j\in  A^k}Q_{ij}^t}$
represent the influence that power injections at non-adjacent areas have 
on the voltage at phase $i$, which, intuitively, should be small compared to influence from power injections in neighboring areas. This is verified for the IEEE 33 and 123 bus test feeders in Fig.~\ref{fig:heatmap} below. Thus, we make the approximation
\begin{equation}
\label{decenlinearflowtimet}
   \Scale[0.8]{ \begin{bmatrix}
         v_i^t\\
         \ \\
         |v_i^t|
     \end{bmatrix} \approx \begin{bmatrix}
        \sum\limits_{j\in A^l}P_{ij}^t+\sum\limits_{k\in\mathcal{N}(l)}\sum\limits_{j\in A^k}P_{ij}^t\\
        \ \\ 
         \sum\limits_{j\in A^l}Q_{ij}^t+\sum\limits_{k\in\mathcal{N}(l)}\sum\limits_{j\in A^k}Q_{ij}^t
     \end{bmatrix} + \begin{bmatrix}
          w_i^t\\
          \ \\ 
         |w_i^t|
     \end{bmatrix}.}
\end{equation}
 By (\ref{decenlinearflowtimet}), for a phase $i$ in areal $l$ ($\Scale[0.9]{1\leq l\leq n_A}$) at $\Scale[0.9]{t=1,\ldots, T}$, we have 
\begin{equation}
\label{decenlinearflowcomplex}
       \Scale[0.8]{ \begin{bmatrix}
         v_i^1\\
         |v_i^1|\\
         \vdots\\
          v_i^{T}\\
         |v_i^{T}|\\
     \end{bmatrix} \approx \begin{bmatrix}
        \sum\limits_{j\in A^l}P_{ij}^1+\sum\limits_{k\in\mathcal{N}(l)}\sum\limits_{j\in A^k}P_{ij}^1\\
         \sum\limits_{j\in A^l}Q_{ij}^1+\sum\limits_{k\in\mathcal{N}(l)}\sum\limits_{j\in A^k}Q_{ij}^1\\
         \vdots\\
        \sum\limits_{j\in A^l}P_{ij}^{T}+\sum\limits_{k\in\mathcal{N}(l)}\sum\limits_{j\in A^k}P_{ij}^{T}\\
         \sum\limits_{j\in A^l}Q_{ij}^{T}+\sum\limits_{k\in\mathcal{N}(l)}\sum\limits_{j\in A^k}Q_{ij}^{T}
     \end{bmatrix} + \begin{bmatrix}
          w_i^1\\
         |w_i^1|\\
         \vdots\\
          w_i^{T}\\
         |w_i^{T}|\\
         \end{bmatrix}.}
\end{equation}
Set \[ \Scale[0.8]{
y_i = \begin{bmatrix}
         v_i^1\\
         |v_i^1|\\
         \vdots\\
          v_i^{T}\\
         |v_i^{T}|\\
     \end{bmatrix}, T_{k,i}=\begin{bmatrix}
    \sum\limits_{j\in A^k} P_{ij}^1\\
     \sum\limits_{j\in A^k} Q_{ij}^1\\
     \vdots\\
         \sum\limits_{j\in A^k} P_{ij}^{T}\\
     \sum\limits_{j\in A^k} Q_{ij}^{T}\\
\end{bmatrix}, \ \text{and}\ b_i =  \begin{bmatrix}
          w_i^1\\
         |w_i^1|\\
         \vdots\\
          w_i^{T}\\
         |w_i^{T}|\\
         \end{bmatrix}.}
\]
The approximate decentralized linear model (\ref{decenlinearflowcomplex})  can be then written as 
\begin{equation}
    \label{decenlinearflow}
    \Scale[0.85]{y_i \approx T_{l,i}+\sum\limits_{k\in\mathcal{N}(l)}T_{k,i}+b_i.}
\end{equation}
With this approximation, every area $k$ must communicate the term 
$\Scale[0.90]{T_{k,i}}$
for all buses $i$ in neighboring areas. 
The process is outlined in Algorithm 
\ref{algorithm:decentralizedpowerflow}.
}

\begin{algorithm}
 \caption{Algorithm to Calculate Voltage Phasors from Decentralized Linear Load Model at Area $\Scale[0.90]{l}$ ($\Scale[0.90]{1\leq l\leq n_{_A}}$)}
 \label{algorithm:decentralizedpowerflow}
 \begin{algorithmic}[1]
 \renewcommand{\algorithmicrequire}{\textbf{Input:}}
 \renewcommand{\algorithmicensure}{\textbf{Output:}}
 \REQUIRE Model matrix $\Scale[0.90]{M}$
 \ENSURE  Estimated $y_i$ for each bus $\Scale[0.90]{i \in A^l}$ 
  \FOR {$\Scale[0.90]{i\in A^l, k\in \mathcal{N}(l)}$}
  \STATE Send $\Scale[0.90]{T_{k,i}}$
  to area $k$.
  \ENDFOR
   
   \STATE Wait to receive $\Scale[0.90]{T_{k,i}}$ for all $\Scale[0.90]{k\in \mathcal{N}(l), i\in A^l}$
   
   \FOR {$i\in A^l$}
  \STATE  Calculate (\ref{decenlinearflow}) for $y_i$
  \ENDFOR
   
 \RETURN $\Scale[0.90]{y_i$, $\forall i\in A^l}$ 
 
 \end{algorithmic} 
 \end{algorithm}
{For notational simplicity, we introduce the following notations. Let $c_{_{1}}, c_{_{2}}, \ldots c_{_{n_l}}$ denote the indices of phases in area $l$.  We define the linear maps $\Scale[0.9]{E_{_{lk}}(X)}$ and $\Scale[0.9]{E_{_{ll}}(X)}$ as 
\[\Scale[0.85]{
E_{_{lk}}(X)=-\begin{bmatrix}
T_{_{k,c_1}}\\
\vdots\\
T_{_{k,c_l}}
\end{bmatrix}\ 
\text{and} \  
E_{_{ll}}(X)= \begin{bmatrix}
    y_{_{c_{_1}}}\\ \vdots \\ y_{_{c_{_{n_l}}}}
\end{bmatrix}
-\begin{bmatrix}
T_{_{l,c_{_1}}}\\
\vdots\\
T_{_{l,c_{_{n_l}}}}
\end{bmatrix}}.
\]
Additionally, we define the vector $\Scale[0.85]{f_l}$ as $\Scale[0.85]{{f_l=[b_{_{c_{1}}}^{\intercal} \ \ldots \ b_{_{c_{{n_l}}}}^{\intercal}]^{\intercal}}.}$ Using theses expressions, the decentralized linear load flow model at area $l$ can be expressed as
\begin{equation}
\label{linearpowerconstraint}
\Scale[0.85]{
    E_{_{ll}}(X)+\sum_{k\in \mathcal{N}(l)} E_{_{lk}}(X) \approx f_{_{l}}.}
\end{equation}
}
\section{Decentralized Matrix Completion}

\subsection{Formulation}
We adapt a decentralized matrix completion algorithm for sparsity regulated matrix completion from \cite{mardani_mateos_giannakis_2013} to our state estimation formulation.  The algorithm takes advantage of the following characterization of the nuclear norm:
{
\begin{equation}
\label{eqn:nucleareqFrobenius}
\Scale[0.9]{
\begin{aligned}
||X||_*:=& \underset{U \in \mathbb{R}^{m \times r},V \in \mathbb{R}^{r \times n}}{\text{minimize}}\ \frac{1}{2} (||U||_F^2+||V||_F^2) \\
& \quad\ \  \text{s.t.}\ \  \quad\quad\quad\ 
 X=UV,
\end{aligned}}
\end{equation}
where  $r$ is an upper bound on the rank of $\Scale[0.90]{X}$.}
Note that for sufficiently small $r$, using this characterization of the nuclear norm also allows us to dramatically reduce the size of the problem.  We now use \eqref{eqn:nucleareqFrobenius} to reformulate (\ref{centralized_SE}) as follows:
\begin{equation}
\begin{aligned}
\label{step1}
&\text{minimize}\ 
\Scale[0.9]{\frac{1}{2}(||U||_F^2+||V||_F^2)+\frac{\mu}{2} ||{P}_\Omega(UV-M)||_F^2}\\
&\quad\quad\quad\quad\ \Scale[0.9]{+\frac{\nu}{2}||f_1(UV)-(Af_2(UV)+b)||_2^2,}\\
&\text{over} \quad\quad\  \Scale[0.9]{U\in\mathbb{R}^{m\times r}, V\in\mathbb{R}^{r\times n},}\\
&\text{s.t.}\quad\quad\quad \Scale[0.83]{f_1(UV) = \left[
    a_1^{\intercal}UV \ \, 
    a_2^{\intercal}UV \, 
    \ldots \,
    a_{_{2T-1}}^{\intercal}UV\ \,
    a_{_{2T}}^{\intercal}UV
\right]^{\intercal}}, \\
&  \quad\quad\quad\ \ \   \Scale[0.83]{f_2(UV) = \left[
   c_1^{\intercal}UV\ \,
     c_2^{\intercal}UV \,
    \ldots \,
    c_{_{2T-1}}^{\intercal}UV\ \,
    c_{_{2T}}^{\intercal}UV
         \right]^{\intercal}}.
\end{aligned}\end{equation}

Solving this formulation with an alternating method is an effective matrix completion technique, combining benefits of both nuclear norm minimization and alternating minimization \cite{Hastie:2015:MCL:2789272.2912106,rennie_srebro_2005}. When solving (\ref{step1}), this allows us to formulate a scalable algorithm for minimizing the nuclear norm of the matrix without solving a semidefinite program. 

While providing algorithms with significantly lower computational complexity, methods based upon low rank factorization of the form \eqref{eqn:nucleareqFrobenius}  have nonconvex objective functions. 
In general, there is no guarantee that stationary points of a non-convex problem will coincide with the global optima.  However, recent works analyzing the optimization landscape of similar functions have shown that all local minima are also global minima \cite{Ge2016,Josz2018ATO,mardani_mateos_giannakis_2013}.   A similar result is shown for our method below (see Theorem \ref{globalminthm}).

With the low rank factorization, we can now separate the objective function into a sum of functions each utilizing only some columns of $\Scale[0.9]{V}$, allowing for decentralized computations.  As in \cite{mardani_mateos_giannakis_2013}, the basis matrix, $\Scale[0.9]{U}$, must be the same in every control area for our decentralized method. For the coefficients matrix, $\Scale[0.9]{V}$, each column refers to data from a specific bus, and so we separate the coefficient matrix into matrices for each area with columns corresponding to the buses in that area.

At each area $\Scale[0.9]{l}$ ($\Scale[0.9]{l=1,\ldots,n_{_A}}$), we define the basis matrix to be $\Scale[0.85]{U_{_l}}$ satisfying $\Scale[0.85]{U_{_l}=U}$, and the coefficient matrix to be $\Scale[0.85]{V_{l}\in \mathbb{R}^{r \times n_l}}$ satisfying that $\Scale[0.85]{V=[V_{1}\ \cdots \ V_{{n_A}}]}$, where $n_{_l}$ is the number of phases at area $l$. %
We now reformulate the  {objective function in \eqref{step1} to be decentralized.
With these new decentralized variables, we have 
$\Scale[0.85]{||U||_F^2=\frac{1}{n_{_A}}\sum_{l=1}^{n_{_A}} ||U_l||_F^2}$ and $\Scale[0.85]{||V||_F^2=\sum_{l=1}^{n_{_A}} ||V_l||_F^2},$ so the first term in  the objective function of \eqref{step1}} can be written as $\Scale[0.85]{\sum_{l=1}^{n_{_A}} \frac{1}{2} \left(\frac{1}{n_A}||U_l||_F^2+||V_l||_F^2\right).}$
{Define $\Scale[0.85]{M= [M_1\ \cdots \  M_{n_{_A}}]}$ with $\Scale[0.85]{M_l\in\mathbb{R}^{m\times n_l}}$ and 
$\Scale[0.85]{\Omega_l}$ as the set including the known elements in $\Scale[0.85]{M_l}$.
Then  the second term in the objective function of (\ref{step1}) becomes $\Scale[0.85]{\frac{\mu}{2}\sum_{l=1}^{n_A} ||{P}_{\Omega_l}(U_lV_l-M_l) ||_F^2}$. Using (\ref{linearpowerconstraint}), the linear load flow term (the last term) in the objective function of (\ref{step1}) can be written as $\Scale[0.85]{\frac{\nu}{2}\sum_{l=1}^{n_A}||E_{ll}(UV)+\sum_{j\in \mathcal{N}(l)}E_{lj}(UV)-f_l||_2^2}$. Therefore, (\ref{step1}) can be written in a decentralized fashion as:}
\begin{equation}
\label{step2}
\begin{aligned}
 &\Scale[0.85]{\underset{\{U_l\}, \{V_l\}}{\text{minimize}}\ 
\sum_{l=1}^{n_A} \frac{1}{2} \left(\frac{1}{n_A}||U_l||_F^2+||V_l||_F^2\right)} \Scale[0.85]{+\sum_{l=1}^{n_A}\frac{\mu}{2}||{P}_{\Omega_l}(U_lV_l-M_l) ||_F^2}\\
&\quad\quad\quad\ \  \Scale[0.85]{+\sum_{l=1}^{n_A}\frac{\nu}{2}||E_{ll}(UV)+\sum_{j\in \mathcal{N}(l)}E_{lj}(UV)-f_l||_2^2},\\
&\ \text{s.t.}\
 \ \  \ \ \ \Scale[0.85]{U_l=U_j , \;  \forall  j \in \mathcal{N}(l).}
\end{aligned}
\end{equation}

When developing an algorithm to solve problem (\ref{step2}),  we desire not to utilize variables from other areas at a given area, which prompts us to  introduce an auxiliary variable for the equality constraint, as done in  \cite{DistributedSDP,mardani_mateos_giannakis_2013}.  We introduce $\Scale[0.85]{S_{lj}=U_l=U_j}$ to ensure they have the same basis if two areas are neighbours and $\Scale[0.85]{q_{lj}=E_{lj}(UV)}$ to denote the term being communicated in the decentralized linear model. 
With the new variables introduced, (\ref{step2}) becomes
\begin{equation}
\label{eqn:fullDecentralizedModel}
\begin{aligned}
 &\Scale[0.85]{\underset{\{U_l\}, \{V_l\}}{\text{minimize}}\ 
\sum_{l=1}^{n_A} \frac{1}{2} \left(\frac{1}{n_A}||U_l||_F^2+||V_l||_F^2\right)+\sum_{l=1}^{n_A}\frac{\mu}{2}||{P}_{\Omega_l}(U_lV_l-M_l) ||_F^2}\\
&\quad\quad\quad\quad \Scale[0.85]{+\sum_{l=1}^{n_A}\frac{\nu}{2}||E_{ll}(UV)+\sum_{j\in \mathcal{N}(l)}q_{lj}-f_l||_2^2},\\
&\ \text{s.t.}\
 \ \  \ \ \ \Scale[0.85]{U_l=S_{lj}=U_j , \;  \forall  j \in \mathcal{N}(l)},\\
 &\quad\quad\ \ \ \   \Scale[0.85]{E_{lj}(UV)=q_{lj},\forall  j \in \mathcal{N}(l)}.
\end{aligned}
\end{equation}

We present a theorem adapted from \cite{mardani_mateos_giannakis_2013} to show that any stationary point of (\ref{eqn:fullDecentralizedModel}), can be used to construct a global minimum of (\ref{step1}).  
{Define the linear map $\Scale[0.85]{\mathcal{B}: \mathbb{R}^{m \times n} \rightarrow \mathbb{R}^L (L= |\Omega|+n_{_A})}$  and vector $\Scale[0.85]{d \in \mathbb{R}^L}$ as follows
\begin{equation}
\label{linearexpression}
\Scale[0.72]{\mathcal{B}(X)=\begin{bmatrix}
X_{i_1, j_1}\\
\vdots\\
X_{i_{_{|\Omega|}}, j_{_{|\Omega|}}}\\
\sqrt{\frac{\nu}{\mu}} \big(E_{_{11}}(X)+\sum\limits_{j\in \mathcal{N}(1)} E_{_{1j}}(X) \big)\\
\vdots \\
\sqrt{\frac{\nu}{\mu}} \big(E_{n_{_A}n_{_A}}(X)+\sum\limits_{j\in \mathcal{N}(n_{_A})} E_{_{n_{_A}j}}(X) \big)\\
\end{bmatrix}}\ \text{and} \ 
\Scale[0.72]{d=\begin{bmatrix}
M_{i_{_1}, j_{_1}}\\
\vdots\\
M_{i_{_{|\Omega|}}, j_{_{|\Omega|}}}\\
\sqrt{\frac{\nu}{\mu}}f_1\\
\vdots\\
\sqrt{\frac{\nu}{\mu}}f_{n_{_A}}
\end{bmatrix}},
\end{equation}
where $\Scale[0.85]{(i_{_1},j_{_1}),\ldots, (i_{_{|\Omega|}},i_{_{|\Omega|}})}$ denote the entries in $\Scale[0.85]{\Omega}$. Then the last two terms in the objective function of (\ref{eqn:fullDecentralizedModel}) equal $\Scale[0.85]{\frac{\mu}{2}||\mathcal{B}(X) -d||_2^2}$.

{Define $\Scale[0.85]{\langle B,X\rangle _F = \text{trace}(B^{\intercal}X)}$ for $\Scale[0.85]{B,X\in\mathbb{R}^{m\times n}}$.} 
Before stating the theorem, we define the adjoint of $\Scale[0.85]{\mathcal{B}}$, denoted {by} $\Scale[0.85]{\mathcal{B}^*}$, {as follows:}
if $\Scale[0.85]{[\mathcal{B}(X)]_i = \langle B_i, X \rangle_F}$, then $\Scale[0.85]{\mathcal{B}^*: \mathbb{R}^L \rightarrow \mathbb{R}^{m \times n}}$ is defined as  
$\Scale[0.85]{\mathcal{B}^*(z)= \sum_i z_i B_i,}$ where $z$ is a vector in $\Scale[0.85]{\mathbb{R}^L}$. }
{\begin{thm}
\label{globalminthm}
         Let $\Scale[0.85]{\{\overline{U}_l\},\{ \overline{V}_l\}}$ be a stationary point of (\ref{eqn:fullDecentralizedModel}). 
         Then $\Scale[0.85]{\overline{U}=\overline{U}_1=\cdots=\overline{U}_{n_A}}$ and $\Scale[0.85]{\overline{V}=[\overline{V}_1\ \cdots \ \overline{V}_{n_A}]}$ satisfy that 
         the matrix $\Scale[0.85]{\overline{X}=\overline{U}\overline{V}}$ is a global minimum of
\begin{equation} \label{eqn:thm_eqn}
\Scale[0.9]{
\begin{aligned}
& \underset{X}{\text{minimize}}
& & ||X||_*  +\frac{\mu}{2} ||\mathcal{B}(X)-d||_2^2,
\end{aligned}}
\end{equation}if $\Scale[0.9]{||\mu \mathcal{B}^*(\mathcal{B}(\overline{X})-d) |||_2 \leq 1}$ holds.\end{thm}
The proof of Theorem \ref{globalminthm} is shown in Appendix A.} Note that we can choose $\mu$ small enough to ensure the condition of Theorem~\ref{globalminthm} is met. In the next section, we propose an algorithm to seek stationary points of \eqref{eqn:fullDecentralizedModel}. 


\subsection{A Proximal ADMM Algorithm}
\label{section:ADMM}
Alternating Direction Method of Multipliers (ADMM) has been shown to be a useful technique in decentralized algorithms \cite{DistributedSDP,distOptPowFlow}.  While lacking convergence guarantees in the general nonconvex setting, ADMM has been applied to bi-convex problems such as Nonnegative Matrix Factorization with computational success \cite{nmf}. Moreover, by adding  a proximal term in the traditional ADMM approach, the \emph{proximal ADMM} method, has been shown to converge for some nonconvex problems \cite{proxADMM}. We next develop the proximal ADMM method to solve \eqref{eqn:fullDecentralizedModel}.

To this end, define the \emph{scaled} form of the augmented  Lagrangian of (\ref{eqn:fullDecentralizedModel}) as follows:
\begin{align}
\label{eqn:Lagrangian}
&\Scale[0.9]{\mathcal{L}(\{U_{_l}\},\{V_{_l}\},\{S_{_{lj}}\}, \{q_{_{lj}}\},\{\Gamma_{_{lj}}\}, \{\Lambda_{_{lj}}\})}\nonumber\\
&\Scale[0.9]{:= \sum_{l=1}^{n_{_A}} \frac{1}{2} \left(\frac{1}{n_{_A}}||U_l||_F^2+||V_l||_F^2+\mu||{P}_{\Omega_l}(U_{_l}V_{_l}-M_{_l}) ||_F^2\right)}\nonumber\\
&\ \  \ \Scale[0.9]{ +\sum_{l=1}^{n_{_A}}\frac{\nu}{2}||E_{_{ll}}(UV)+\sum_{j\in \mathcal{N}(l)}q_{_{lj}}-f_{_l}||^2}\nonumber\\
&\ \ \  \Scale[0.9]{+ \sum_{l=1}^{n_{_A}}\sum_{j \in \mathcal{N}(l)} \frac{\gamma}{2}||U_{_l}-S_{_{lj}}+\Gamma_{_{lj}}||_F^2}\nonumber\\
&\ \ \ \Scale[0.9]{+\sum_{j=1}^{n_{_A}}\sum_{l \in \mathcal{N}(j)}\frac{\lambda}{2}||q_{_{lj}}-E_{_{lj}}(UV)+\Lambda_{_{lj}}||_F^2,}
\end{align}
where $\Scale[0.85]{\Gamma, \Lambda}$ are dual variables, and $\Scale[0.85]{\mu, \nu, \gamma, \lambda}$ are fixed constants.
Note that we are utilizing the scaled form of the augmented Lagrangian as presented by  \cite{Boyd_admm}. 

The proximal ADMM method then proceeds iteratively by minimizing over each primal variable individually, and then maximizing over the dual variables.
In particular, at each area $l$, the minimization of the {proximal} Lagrangian over $\Scale[0.85]{U_l}$ and $\Scale[0.85]{V_l}$  is performed at each iteration $k$ according to:
\begin{equation}
    \label{eqn:UV_Update}
\begin{split}
    \Scale[0.9]{U_l^{(k+1)}= \underset{U_l}{\text{argmin }}\mathcal{L}_{U}^{(k)}(U_l) + \frac{c}{2}||U_l-U_l^{(k)}||_F^2,} \\
\Scale[0.9]{V_l^{(k+1)}= \underset{V_l}{\text{argmin }}\mathcal{L}_{V}^{(k)}(V_l) + \frac{c}{2}||V_l-V_l^{(k)}||_F^2.}
\end{split}
\end{equation}
where the last term in each update is the proximal term, $c > 0$ is a parameter, and
 we have used the following shorthand notation:
\begin{align*}
    \Scale[0.85]{\mathcal{L}_{U}^{(k)}(U_l):=
     \mathcal{L}(U_l,\{U_j^{(k)}\}_{j\neq l},\{V_{_l}^{(k)}\},\{S_{_{lj}}^{(k)}\},
\{q_{_{lj}}^{(k)}\},\{\Gamma_{_{lj}}^{(k)}\}, \{\Lambda_{_{lj}}^{(k)}\})}
\end{align*}
and
\begin{align*}
    \Scale[0.85]{\mathcal{L}_{V}^{(k)}(V_{_l}):=\mathcal{L}(\{U_{_l}^{(k+1)}\},V_{_l},\{V_{_j}^{(k)}\}_{j\neq l},\{S_{_{lj}}^{(k)}\},
\{q_{_{lj}}^{(k)}\},\{\Gamma_{_{lj}}^{(k)}\}. \{\Lambda_{_{lj}}^{(k)}\}).}
\end{align*}


To update the auxiliary variables and dual variables, each area communicates the basis and the coefficients corresponding to shared entries with all neighboring areas.  A closed form solution for the the auxiliary variables can be obtained through the first order necessary condition. The auxiliary variable $\Scale[0.85]{S_{lj}}$ for $\Scale[0.85]{j\in \mathcal{N}(l)}$  is updated as follows
\begin{equation}
\label{eqn:auxupdate}
 \Scale[0.9]{S^{(k+1)}_{_{lj}}=\frac{1}{2}(U^{(k+1)}_{_l}+U^{(k+1)}_{_j}).}
\end{equation}
At each area $l$, the auxiliary variables $q_{_{lj}}$ for $\Scale[0.85]{j\in \mathcal{N}(l)}$ are updated simultaneously by solving the following system of linear equations:
\begin{equation}
\begin{aligned}
&\Scale[0.9]{q_{_{lj}}^{(k+1)}=\frac{1}{\lambda+\nu}\left[\lambda\left(E_{_{lj}}\left(U^{(k+1)}{V^{(k+1)}}\right)-\Lambda_{_{lj}}^{(k)}\right)\right.}+\\
&\quad\quad\Scale[0.9]{\left.\nu\left(f_l-E_{_{ll}}\left(U^{(k+1)}V^{(k+1)}\right)-\sum_{i\in \mathcal{N}(l), i \neq j} q^{(k+1)}_{li}\right)\right].} 
\end{aligned}
\label{eqn:q_update}
\end{equation}
Finally, the dual variables are updated as 
\begin{equation}
\label{eqn:dualupdate}
    \begin{split}
         \Scale[0.9]{\Gamma^{(k+1)}_{_{lj}}=\Gamma^{(k)}_{_{lj}}+ U^{(k+1)}_l -S^{(k+1)}_{_{lj}},}\\
    \Scale[0.9]{\Lambda^{(k+1)}_{_{lj}}=\Lambda^{(k)}_{_{lj}}+\big(q^{(k+1)}_{_{lj}}-E_{_{lj}}\big(U^{(k+1)}{V^{(k+1)}}\big)\big).}
    \end{split}
\end{equation}
The algorithm is summarized in Algorithm \ref{algo:main}.

 \begin{algorithm} 
 \caption{Decentralized Proximal ADMM Algorithm for State Estimation at Area $l$}
 \begin{algorithmic}[1] \label{algo:main}
 \renewcommand{\algorithmicrequire}{\textbf{Input:}}
 \renewcommand{\algorithmicensure}{\textbf{Output:}}
 \REQUIRE $\Scale[0.85]{A_l, b_l, E_{ll}, E_{jl}}$ for any $\Scale[0.85]{j \in \mathcal{N}(l), f_l}$ 
 \ENSURE  Completed matrix for area $\Scale[0.85]{l, X_l}$
 \\ \textit{Initialization} : $\Scale[0.85]{\Gamma_{lj}=0, \Lambda_{lj}=0}$ for any $\Scale[0.85]{j \in \mathcal{N}(l)}$
 
  \FOR {$k= 1,\ldots, N$} 
  \STATE Solve (\ref{eqn:UV_Update}) for $\Scale[0.85]{U_l^{(k)}, V_l^{(k)}}$ 
  \STATE \textbf{Send:} $\Scale[0.85]{U_l^{(k)}, E_{jl}(U^{(k)}V^{(k)})}$ to all neighbors $\Scale[0.85]{j \in \mathcal{N}(l)}$
  \STATE \textbf{Receive:} $\Scale[0.85]{U_i^{(k)}, E_{lj}(U^{(k)}V^{(k)})}$ from all neighbors $\Scale[0.85]{j \in \mathcal{N}(l)}$
  \STATE Use \eqref{eqn:q_update} to find $\Scale[0.85]{q^{(k)}_{lj}}$ for $\Scale[0.85]{j\in \mathcal{N}(l)}$ 
  \FOR{$\Scale[0.85]{\forall j \in \mathcal{N}(l)}$}
  \STATE Update $\Scale[0.85]{S^{(k)}_{lj},\Gamma^{(k)}_{lj}}$, and $\Scale[0.9]{\Lambda^{(k)}_{lj}}$ using (\ref{eqn:auxupdate}) and (\ref{eqn:dualupdate})\\
  \ENDFOR
\STATE \textbf{Send:} $q_{lj}^{(k)}$ to all neighbors $\Scale[0.85]{j \in \mathcal{N}(l)}$
  \STATE \textbf{Receive:} $q_{jl}^{(k)}$ from all neighbors $\Scale[0.85]{l \in \mathcal{N}(j)}$
  \ENDFOR
 \RETURN $\Scale[0.85]{X_l=U_l^{(N)}V_l^{(N)}}$ 
 \end{algorithmic} 
 \label{Agorthim:decenralizedAltMin}
 \end{algorithm}

Note that the information that must be communicated is relatively small.  If each bus has $m$ measurements associated with it, then neighboring areas $l$ and $j$ must communicate $n_l+n_j+mr$ measurements, where $n_i$ is the number of phases in area $i$.  This quantity is much smaller compared to the total $(n_l+n_j)m$ measurements at both areas.

Finally, it is worth pointing out that while the proof of convergence of Algorithm \ref{algo:main} remains an open research question, in our numerical experiments, the algorithm converged to stationary points of \eqref{eqn:fullDecentralizedModel}, which, by the virtue of Theorem \ref{globalminthm}, are global minima of the original problem \eqref{centralized_SE}.


\section{Numerical Results}
\label{section:numericalResults}
\label{section:decentralized}
\begin{figure}
   \centering
   \includegraphics[scale=0.295,trim= 30 120 25 50, clip]{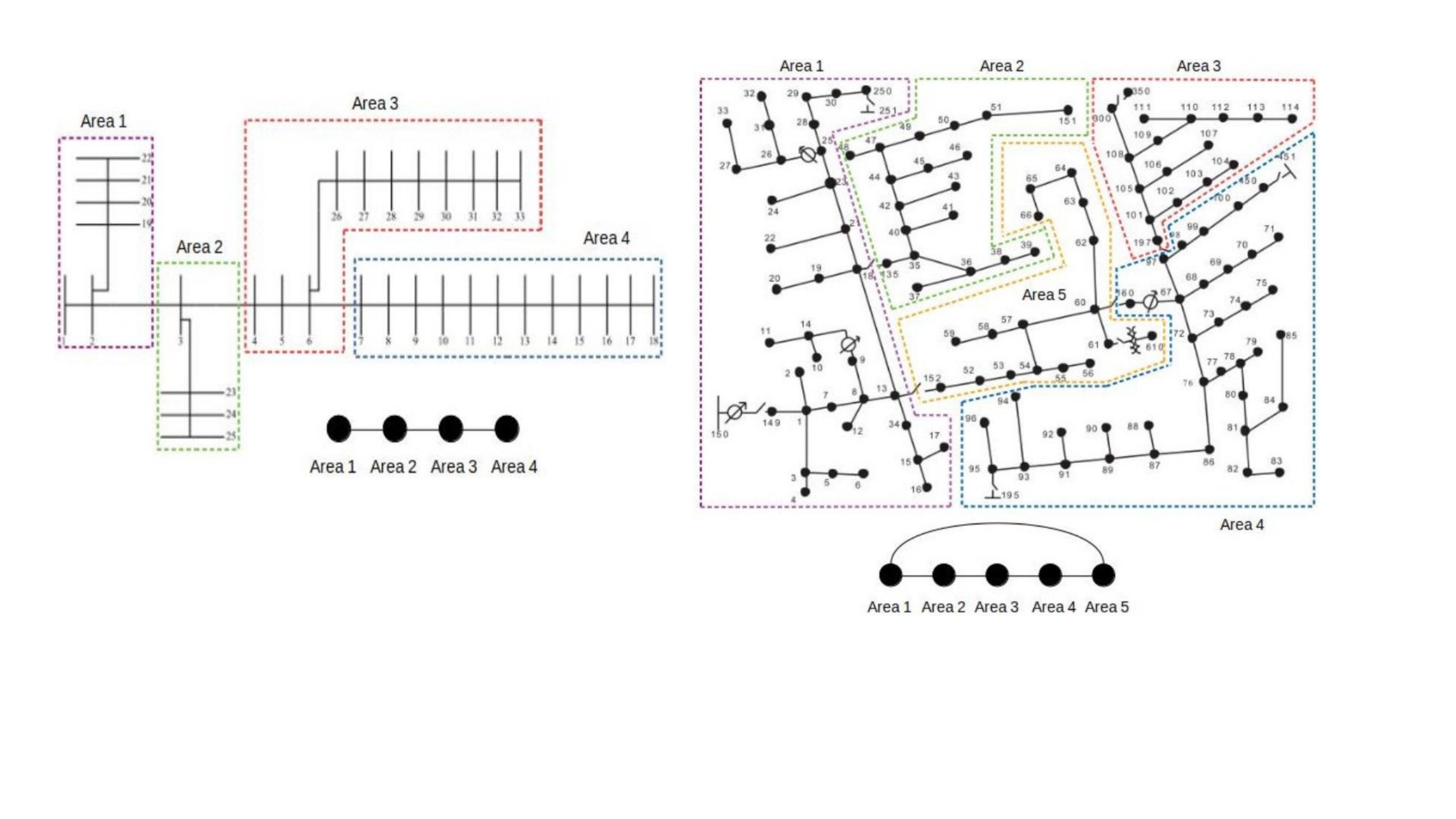}
   \caption{(a) Decentralized IEEE 33 bus system (b) Decentralized IEEE 123 bus system}
   \label{fig:ieeesystems}
\end{figure}
We demonstrate the performance of the decentralized linear load flow model on the IEEE 33 and 123 bus test feeders, and evaluate the decentralized proximal ADMM algorithm on the IEEE 123 bus test feeder.  The IEEE 33 bus test feeder is a balanced network modified to have solar panels
added at buses 16, 23, and 31, and the IEEE 123 bus test feeder is a multi-phase unbalanced radial distribution system, in which buses are single-, double-, or three-phase (with 263 phases in total).  We seperate the IEEE 33 bus test feeder into four areas as shown in Fig.~\ref{fig:ieeesystems}(a), and the IEEE 123 bus test feeder into five areas as shown in Fig.~\ref{fig:ieeesystems}(b). 
For results using three paritions of the IEEE 33 bus test feeder, we combine areas two and three in Fig.~\ref{fig:ieeesystems}(a) into one area.  For results that use two, three, or four partitions of the IEEE 123 bus test feeder, we combine the areas in Fig.~\ref{fig:ieeesystems}(b) as follows: For the two area partition, we combine areas two, three, four, and five into one area.  For the three area partition, we combine areas two and three into one area and areas four and five into one area.  For the four area partition, we combine areas one and two.  
\subsection{Decentralized Linear Model}
\begin{figure}
   \centering
   \includegraphics[scale=0.45,trim= 40 240 45 105, clip]{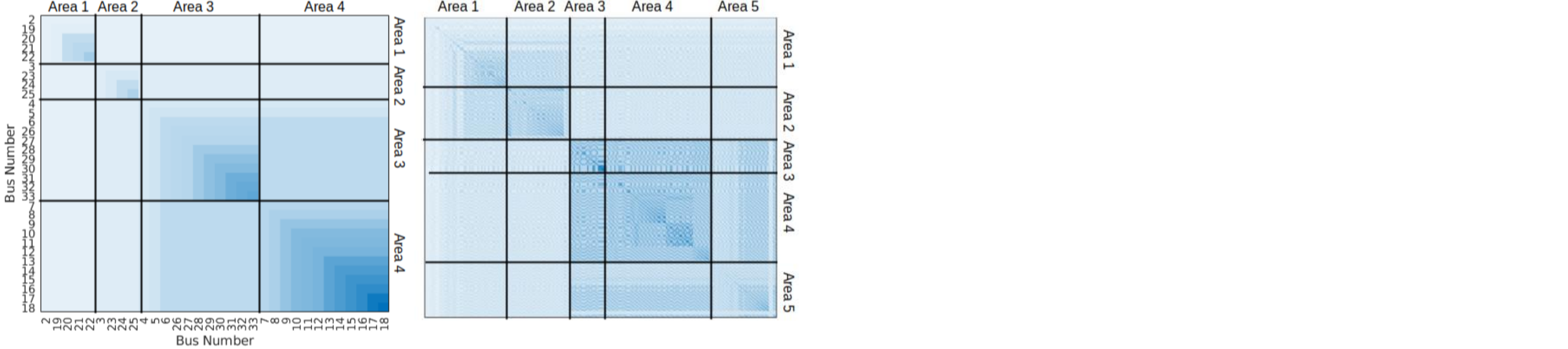}
   \caption{(a) IEEE 33 bus system (b) IEEE 123 bus system \ Visual representation of matrix $\Scale[0.9]{N}$ for the IEEE 33 and 123 bus systems.  The matrices are rearranged and sorted by areas to demonstrate the approximate block structure.}
   \label{fig:heatmap}
\end{figure}
We start by numerically evaluating the decentralized approximation to the linear load flow model.  In our simulation setting, $N^1=\cdots=N^{T} =N$ in (\ref{eqn:linvol}).
 The first $\Scale[0.85]{|\mathcal{P}|}$ columns of  matrix $\Scale[0.85]{N}$ from equation \eqref{eqn:linvol} are shown in Fig.~\ref{fig:heatmap}. 
We see that this matrix has a block structure allowing us to make the truncation necessary for the decentralized method.  If the areas are chosen such that buses $i$ and $j$ where $\Scale[0.85]{N_{ij}}$ is large, that is, power input at bus $j$ is highly correlated with voltage at bus $i$, are in the same area or neighboring areas.

Analytically, we propose a metric to predict the accuracy of the decentralized model for a given partitioning of buses into areas. 
The matrix $\Scale[0.85]{N}$ in equation (\ref{eqn:linvol}) is approximated by the matrix $\Scale[0.85]{\tilde{N}}$ where $\Scale[0.85]{\tilde{N}_{ij}=N_{ij}}$ if the buses corresponding to rows $i$ and $j$ are in the same area or neighboring areas, and  $\Scale[0.85]{\tilde{N}_{ij}=0}$ otherwise.  With this, the decentralized model is equivalent to using the original model with this new matrix.  Thus, we use the relative Frobenius norm difference $\Scale[0.85]{||N-\tilde{N}||_F/||N||_F}$ to measure the quality of the approximation. For the four area 33 bus system shown in Fig.~\ref{fig:ieeesystems}(a), the relative error is 0.0411, and for the five area 123 bus system, the relative error is 0.2604.

We evaluate the effectiveness of the decentralized linear load flow model compared to the original linear load flow approximation in Fig.~\ref{fig:decentralizedLinLoadFlow}.  As shown, the decentralized method performs similarly to the centralized model.
\begin{figure}
    \centering
  \includegraphics[scale=0.4,trim= 20 15 30 40, clip]{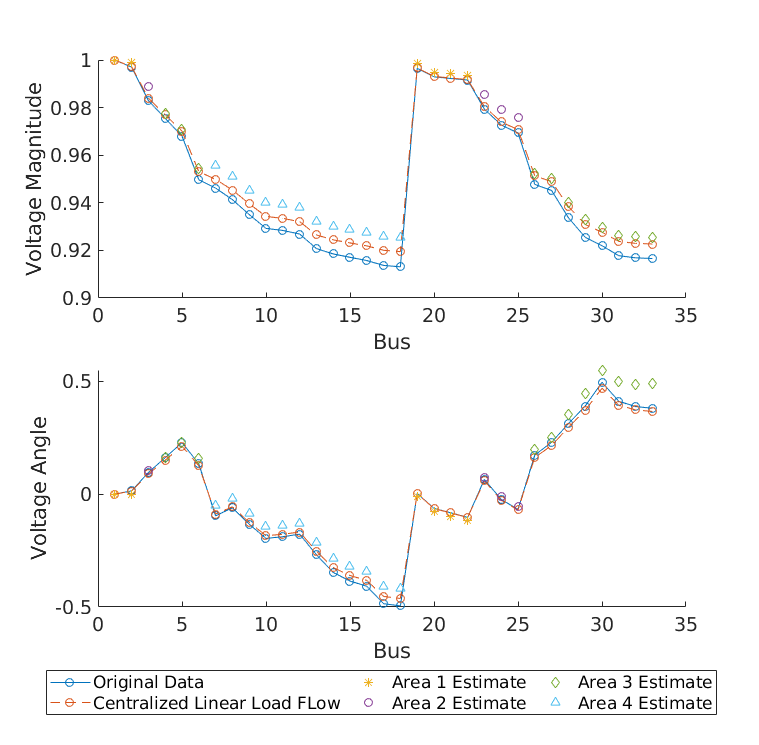}
    \caption{Voltage magnitude and angle reconstructed from the centralized and decentralized linear load flow models.}
    \label{fig:decentralizedLinLoadFlow}
\end{figure}

We use the mean absolute percentage error (MAPE) of voltage magnitude and the mean absolute error (MAE) of voltage angle to measure the accuracy of the linear model for the IEEE 33 and 123 bus feeders under different partitions. The results on the IEEE 33 bus system with one, three, and four area partitions and those on the IEEE 123 bus system with one, three, four, and five area partitions are shown in Table \ref{table:linmodelpartitions}.
 Note that the case of two areas in the 33 bus system and two and three areas in the 123 bus system are omitted as there is no information loss compared to the one area case, and thus the results would be identical to the one area case. While the MAPE of  voltage magnitude and MAE of  voltage angle are notably larger when the network is partitioned into more than three areas, the error is on the same order of magnitude in every case.  
The MAPE of voltage magnitude and MAE of voltage angle for the one, two, and three area results are functionally equivalent (within the 95\% confidence interval), which is expected as the loss functions are equivalent in each case. 

\begin{table}
    \caption{MAPE of voltage magnitude and MAE of voltage angle  from the linear power flow model utilizing the decentralized approximation for the IEEE 33  and 123 bus test feeders}
    \label{table:linmodelpartitions}
\begin{tabular}{ll||l|l|l|l|}
\cline{3-6}
                                                &           & 1 Area & 3 Areas & 4 Areas & 5 Areas               \\ \hline \hline
\multicolumn{1}{|l|}{\multirow{2}{*}{IEEE 33}}  & Magnitude & 0.471       & 0.738   & 0.764   & \multicolumn{1}{c|}{-} \\ \cline{2-6} 
\multicolumn{1}{|l|}{}                          & Angle     & 0.0337     & 0.0379 & 0.0401  & \multicolumn{1}{c|}{-} \\ \hline \hline
\multicolumn{1}{|l|}{\multirow{2}{*}{IEEE 123}} & Magnitude & 0.148       &\multicolumn{1}{c|}{-} & 0.408   & 0.590                \\ \cline{2-6} 
\multicolumn{1}{|l|}{}                          & Angle     & 0.131     & \multicolumn{1}{c|}{-} & 0.141 & 0.218               \\ \cline{1-6} 
\end{tabular}
\end{table}
  \subsection{IEEE 123 bus test feeder}
  The effectiveness of the decentralized proximal ADMM method in Algorithm \ref{Agorthim:decenralizedAltMin} is demonstrated on the multiphase IEEE 123 bus test feeder.  
   We consider the real and imaginary parts of voltage phasor as variables, and the voltage magnitude, active power, and reactive power as potentially known measurements.  
  In our experiments,  1\% Gaussian noise is added to the measurements. Unless stated otherwise,  we assume that the voltage phasors at the slack bus are known and
  50\% of the data could be obtained by SCADA measurements (active, reactive power injections and voltage magnitudes).  Additionally, all results are based on the partition with five areas and the data matrix being built with 5 minutes of data (thus of size $25 \times 260$) unless otherwise stated.  

    Observe that when less than $2/3$ of the potential measurements  are available, the estimation problem is underdetermined; we refer to this case as \emph{low observability}.
     Fig.~\ref{fig:Percent_data_known} demonstrates that the algorithm performs well (less than 2\% MAPE) even when  10\% of the data is available.  
    The MAPE of voltage magnitude and the MAE of voltage angle decrease as more data are available, though the difference is quite small after 50\% of the data are available.
     Due to the less accurate linear model for the 5-area case, the MAPE of voltage magnitude for the 5-area  case is slightly larger than  the 1-area case, but still within 1\% once the data availability reaches 20\%.  The MAE of voltage angle is smaller for the 5-area case when the data availability is between 10\% and 40\%, and it is comparable for the two cases when the data availability is over 40\%. 
    \begin{figure}
    \centering
  \includegraphics[scale=0.55,trim= 0 0 20 20, clip]{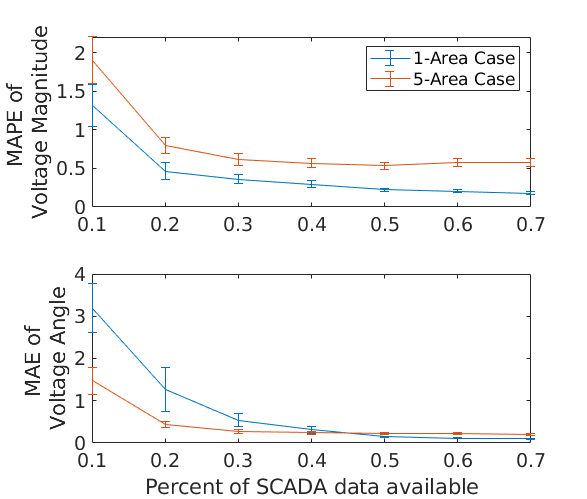}
    \caption{MAPE of voltage magnitude and MAE of voltage angle for 1-area case and 5-area case vs.  percentage of data known 
} 
    \label{fig:Percent_data_known}
\end{figure}
    \begin{table}
\caption{MAPE of voltage magnitude and MAE of voltage angle for  one and five area partitions utilizing 1-, 3-, 5-, and 10-time step with 1\% noise} 
\label{table:Varying_nt}
\center

\begin{tabular}{ll|l|l|l|l|}
\cline{3-6}
                                                  &               & \multicolumn{1}{c|}{$T=1$}                               & \multicolumn{1}{c|}{$T=3$}                               & $T=5$                                                    & \multicolumn{1}{c|}{$T=10$}                              \\ \hline
\multicolumn{1}{|l|}{\multirow{2}{*}{$|v|$}}      & 1-Area   & \begin{tabular}[c]{@{}l@{}}0.328\\ $\pm$0.027\end{tabular} & \begin{tabular}[c]{@{}l@{}}0.213\\ $\pm$0.021\end{tabular} & \begin{tabular}[c]{@{}l@{}}0.182\\ $\pm$0.023\end{tabular} & \begin{tabular}[c]{@{}l@{}}0.149\\ $\pm$0.012\end{tabular} \\ \cline{2-6} 
\multicolumn{1}{|l|}{}                            & 5-Area & \begin{tabular}[c]{@{}l@{}}0.713\\ $\pm$0.046\end{tabular} & \begin{tabular}[c]{@{}l@{}}0.717\\ $\pm$0.033\end{tabular} & \begin{tabular}[c]{@{}l@{}}0.698\\ $\pm$0.041\end{tabular} & \begin{tabular}[c]{@{}l@{}}0.676\\ $\pm$0.032\end{tabular} \\ \hline
\multicolumn{1}{|l|}{\multirow{2}{*}{$\angle v$}} & 1-Area   & \begin{tabular}[c]{@{}l@{}}0.177\\ $\pm$0.016\end{tabular} & \begin{tabular}[c]{@{}l@{}}0.135\\ $\pm$0.02\end{tabular}  & \begin{tabular}[c]{@{}l@{}}0.118\\ $\pm$0.015\end{tabular} & \begin{tabular}[c]{@{}l@{}}0.107\\ $\pm$0.01\end{tabular}  \\ \cline{2-6} 
\multicolumn{1}{|l|}{}                            & 5-Area & \begin{tabular}[c]{@{}l@{}}0.351\\ $\pm$0.028\end{tabular} & \begin{tabular}[c]{@{}l@{}}0.312\\ $\pm$0.016\end{tabular} & \begin{tabular}[c]{@{}l@{}}0.298\\ $\pm$0.026\end{tabular} & \begin{tabular}[c]{@{}l@{}}0.276\\ $\pm$0.022\end{tabular} \\ \hline
\end{tabular}

\end{table}

    Table \ref{table:Varying_nt} demonstrates the advantage of utilizing more time steps in the data matrix.  The simulation was run 20 times with 1\% Gaussian noise and 50\% measurements available, and a 95\% confidence interval is presented.  The trend indicates that utilizing more time steps does result in a more accurate estimation of the voltage phasor, however this trend is more evident in the {1-area case} than in the {5-area case}. 
    We conjecture that this is due to the error in the decentralized linear load flow model being larger, and thus introducing a larger lower bound on the potential performance of the algorithm.
    \begin{table}[]
\center
\caption{MAPE of voltage magnitude and MAE of voltage angles for different partitions on the IEEE 123 bus test feeder}
\label{table:MC_num_areas}
\begin{tabular}{l|l|l|l|l|l|}
\cline{2-6}
                                 & \multicolumn{1}{c|}{1 Area}                                & \multicolumn{1}{c|}{2 Areas}                               & \multicolumn{1}{c|}{3 Areas}                               & \multicolumn{1}{c|}{4 Areas}                               & \multicolumn{1}{c|}{5 Areas}                               \\ \hline
\multicolumn{1}{|l|}{$|v|$}      & \begin{tabular}[c]{@{}l@{}}0.182\\ $\pm$0.023\end{tabular} & \begin{tabular}[c]{@{}l@{}}0.217\\ $\pm$0.023\end{tabular} & \begin{tabular}[c]{@{}l@{}}0.212\\ $\pm$0.022\end{tabular} & \begin{tabular}[c]{@{}l@{}}0.535\\ $\pm$0.036\end{tabular} & \begin{tabular}[c]{@{}l@{}}0.698\\ $\pm$0.041\end{tabular} \\ \hline
\multicolumn{1}{|l|}{$\angle v$} & \begin{tabular}[c]{@{}l@{}}0.118\\ $\pm$0.015\end{tabular} & \begin{tabular}[c]{@{}l@{}}0.126\\ $\pm$0.013\end{tabular} & \begin{tabular}[c]{@{}l@{}}0.111\\ $\pm$0.011\end{tabular} & \begin{tabular}[c]{@{}l@{}}0.23\\ $\pm$0.022\end{tabular}  & \begin{tabular}[c]{@{}l@{}}0.298\\ $\pm$0.026\end{tabular} \\ \hline
\end{tabular}
\end{table}
\begin{figure}
    \centering
  \includegraphics[scale=0.45,trim= 0 0 20 20, clip]{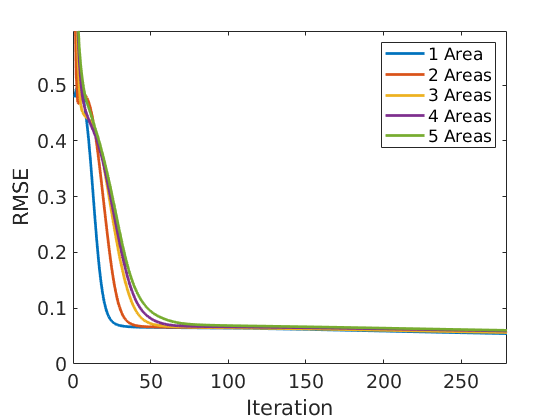}
    \caption{Convergence of Algorithm~\ref{Agorthim:decenralizedAltMin} for different partitions  
    } 
    \label{fig:convergence}
\end{figure}
\begin{figure}
    \centering
  \includegraphics[scale=0.45,trim= 0 0 20 20, clip]{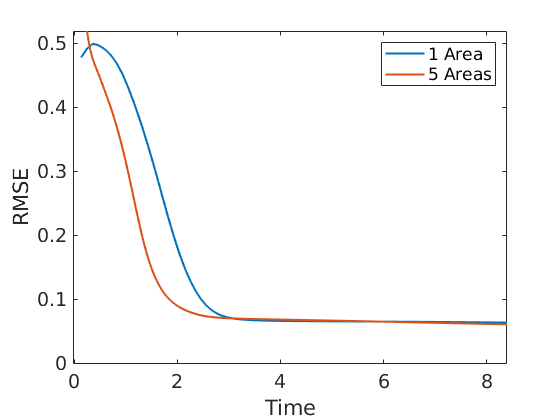}
    \caption{Running time of each iteration for Algorithm~\ref{Agorthim:decenralizedAltMin} with 1-area case (centralized case) and 5-area case}
    \label{fig:time}
\end{figure}

    Next, we analyze the impact of the number of areas on the performance of the algorithm.  While in the decentralized matrix completion algorithm presented in \cite{mardani_mateos_giannakis_2013}, the results  are invariant of the number of areas, our algorithm has additional errors introduced by the linear load flow model when utilizing a larger number of areas.  Table \ref{table:MC_num_areas} demonstrates that while this does contribute to the error, it does so in a similar manner to that observed in Algorithm \ref{algorithm:decentralizedpowerflow} shown in Table \ref{table:linmodelpartitions}.  That is, the error does not propagate in a noticeable way.  The MAPE of voltage magnitude stays below 1\% and the MAE of voltage angle stays below 0.324 degrees for each partition. 
  
Finally, we show the convergence of the algorithm for one through five areas in terms of both iteration number and CPU time.  At each iteration, we record the root mean squared error (RMSE) between the full matrix and the data matrix, and the maximum of  computation time for each area.  We show in Fig.~\ref{fig:convergence} that while the algorithm takes more iterations to converge, the number of iterations is less than 200, and that the rate is not drastically slower when more areas are included.  Additionally, in Fig.~\ref{fig:time}, we show that the overall running time  is more than halved when using five area partitions, as each iteration is considerably cheaper.
\section{Conclusion}
\label{section:conclusions}
We developed a decentralized algorithm for low-observability state estimation problem. The algorithm is based on formulating a regularized matrix completion problem and solving it using the proximal ADMM method. The computation is distributed among different areas in the network, resulting in  less running time compared to the centralized case. 

Although the stationary points of \eqref{eqn:fullDecentralizedModel} are shown to be global minima of the original problem under some conditions (Theorem 1), the proof of convergence of Algorithm \ref{algo:main} to these stationary points remains an interesting research question. A promising direction is to use the recent results from  \cite{Wang2019}
that show convergence of a similar algorithm under some conditions. 
{Additionally, this algorithm can be modified such that voltages can be estimated  in real time \cite{onlineRPCA}}.

\bibliographystyle{unsrt}
\bibliography{references}

\appendix
\subsection{Proof of Theorem \ref{globalminthm}}
We first observe that for every stationary point of (\ref{eqn:fullDecentralizedModel}), the corresponding point $(\overline{U}, \overline{V})$ is a stationary point in the problem 
\begin{equation}
\label{eqn:appendixOneArea}
\Scale[0.85]{
\begin{aligned}
& \underset{U\in\mathbb{R}^{m\times r},V\in\mathbb{R}^{r\times n}}{\text{minimize}}
& & \frac{1}{2} (||U||_F^2+||V||_F^2)+\frac{\mu}{2} ||\mathcal{B}(UV)-d||_2^2,
\end{aligned}
}
\end{equation}
where the linear mapping $\mathcal{B}$ and vector $d$ represent the loss function for each individual area combined into one linear mapping as shown in equation {\eqref{linearexpression}}.  
To see this, we note that 
$$\Scale[0.9]{||\overline{V}||_F^2 =\sum_{i=1}^{n_{_A}} ||\overline{V}_i||_F^2,\  ||\overline{U}||_F^2 =\frac{1}{n_{_A}} \sum_{i=1}^{n_{_A}} ||\overline{U}_i||_F^2},$$
and 
\begin{equation*}
\Scale[0.9]{
\begin{aligned}
\frac{\mu}{2}& ||\mathcal{B}(\overline{U}\overline{V})-d||^2 =\sum_{i=1}^{n_{_A}}   \frac{\mu}{2} ||P_{\Omega_i}(\overline{U}_i\overline{V}_i-M_i)||^2 \\
&\quad \quad\quad+ \sum_{i=1}^{n_{_A}}\frac{\nu}{2} ||{E}_{ii}(\overline{U}_i\overline{V}_i)+\sum_{j \in \mathcal{N}(i)}E_{ij}(\overline{U}_j\overline{V}_j)-f_i||^2.
\end{aligned}
}
\end{equation*} 
\begin{lem}
Let $\Scale[0.9]{(\overline{U}, \overline{V})}$ be a stationary point of (\ref{eqn:appendixOneArea}).  Then,  
\begin{equation}
\begin{split}
&\textup{tr}(\mu \mathcal{B}^*(\mathcal{B}( \overline{U}\overline{V})-d) )^{\intercal} \overline{U}\overline{V})+ \textup{tr}(\overline{V}^{\intercal}\overline{V})=0,\\
&\textup{tr}(\mu \mathcal{B}^*(\mathcal{B}( \overline{U}\overline{V})-d) )^{\intercal} \overline{U}\overline{V})+ \textup{tr}(\overline{U}\overline{U}^{\intercal})=0.
\end{split}
\label{traceequalszero}
\end{equation}
\end{lem}
\begin{proof}
Let $f(U,V)$ denote the objective function in (\ref{step1}).  Any stationary point {$(\overline{U}, \overline{V})$}  must satisfy the first order necessary conditions: 
\begin{align*}
        \nabla_U f(\overline{U},\overline{V})=\mu \mathcal{B}^*(\mathcal{B}( \overline{U}\overline{V})-d)\overline{V}^{\intercal}+ \overline{U} =0_{m \times r},\\
         \nabla_{{V}} f(\overline{U},\overline{V})=\mu \mathcal{B}^*(\mathcal{B}( \overline{U}\overline{V})-d)^{\intercal}\overline{U} +\overline{V}^{\intercal} =0_{n \times r}.
\end{align*}
The lemma can be shown by evaluating $\text{tr}( \nabla_U f(\overline{U},\overline{V})\overline{U}^{\intercal})$ and $\text{tr}( \nabla_V f(\overline{U},\overline{V}) \overline{V})$. Additionally, we have that $\text{tr}(\overline{U}\overline{U}^{\intercal})=\text{tr}(\overline{V}^{\intercal}\overline{V})$.
\end{proof}
We now prove theorem \ref{globalminthm}.
\begin{proof}
By (\ref{eqn:nucleareqFrobenius}), problem (\ref{eqn:appendixOneArea}) is equivalent to the following problem
{\begin{equation}
\label{eqn:appendixoneAreaequiva}
    \underset{U\in\mathbb{R}^{m\times r},V\in\mathbb{R}^{r\times n}}{\text{minimize}} ||X||_*+\frac{\mu}{2} ||\mathcal{B}(UV)-d||_2^2.
\end{equation}
By \cite{exactMatrixCompletion}, problem (\ref{eqn:appendixoneAreaequiva}) is equivalent to 
 the following convex semidefinite programming problem,
}
\begin{equation}
\label{equivSDP}
\begin{aligned}
& \underset{\substack{X\in\mathbb{R}^{m\times n},\\W_1\in\mathbb{R}^{m\times m},\\
W_2\in\mathbb{R}^{n\times n}}}{\text{minimize}}
& & \frac{1}{2}(\text{trace}(W_1)+\text{trace}(W_2))+\frac{\mu}{2}||\mathcal{B}(X)-d||^2\\
& \ \text{s.t.}
& & W = \big[W_1, X; X^{\intercal}, W_2\big] \succeq 0.
\end{aligned}
\end{equation}
The Lagrangian of (\ref{equivSDP}) is given by 
\begin{align*}
   \mathcal{L}&(W_1, W_2, X, {M})=\frac{1}{2}(\text{tr}(W_1)+\text{tr}(W_2)) +\frac{\mu}{2}||\mathcal{B}( X)-d||^2\\
   & -\text{tr}(M_1^{\intercal}W_1+M_3^{\intercal}X^{\intercal})-\text{tr}(M_2^{\intercal}X+M_4^{\intercal}W_2),
\end{align*}
where the dual variable $M$ is decomposed as 
\[
{M}=\begin{bmatrix}{M}_1 & {M}_2\\ {M}_3 & {M}_4\end{bmatrix}.
\]
We now show that $\overline{X}=\overline{U} \overline{V}$, $\overline{W_1}=\overline{U}\overline{U}^{\intercal}$, $\overline{W_2}=\overline{V}^{\intercal}\overline{V}$ is a global optimal of (\ref{equivSDP}), with dual variables  $\overline{M}_1 = \frac{1}{2}I_{m\times m}$, $\overline{M}_4=\frac{1}{2}I_{n\times n}$, $\overline{M}_2=\overline{M}_3^{\intercal} =\frac{\mu}{2} \mathcal{B}^*(\mathcal{B}({\overline{X}})-d)\in\mathbb{R}^{m\times n}$.  We must show four things:  the gradients of the Lagrangian are zero at this point, complementary slackness of the dual variable, primal feasibility, and dual feasibility.

The gradients are given by 
\begin{equation*}
\begin{aligned}
    \nabla_X\mathcal{L} &=\mu \mathcal{B}^*(\mathcal{B}({\overline{X}})-d) -M_2-M_3^{\intercal}, \\
    \nabla_{W_1}\mathcal{L} &= \frac{1}{2}I_{m\times m}-M_1,\\
    \nabla_{W_2}\mathcal{L} &=\frac{1}{2}I_{n\times n}-M_4.
        \end{aligned}
\end{equation*}
These  are  clearly  nulled  by  the  candidate  primal  and  dual variables.
We now show complementary slackness holds.
\begin{equation}
\label{eqn:compslackness}
\begin{aligned}
    \langle \overline{W}, \overline{M}\rangle_F=& \langle \frac{1}{2}I_{m\times m}, \overline{U}\overline{U}^{\intercal} \rangle_F+\langle \frac{1}{2} I_{n\times n}, \overline{V}^{\intercal}\overline{V} \rangle_F\\
    &+ \langle \mu \mathcal{B}^*(\mathcal{B}(\overline{X})-d) , \overline{U}\overline{V}\rangle_F.
    \end{aligned}
\end{equation}

By equation (\ref{traceequalszero}), (\ref{eqn:compslackness}) becomes 
\begin{equation}
    \frac{1}{2}(||\overline{U}||_F^2+||\overline{V}||_F^2)-||\overline{U}||_F^2=0.
\end{equation}


Primal feasibility is shown by 
\begin{equation}
    \overline{W}= \begin{bmatrix}
        \overline{U}\overline{U}^{\intercal}& \overline{U}\overline{V}\\ \overline{V}^{\intercal}\overline{U}^{\intercal}& \overline{V}^{\intercal}\overline{V}
    \end{bmatrix} = \begin{bmatrix}
        \overline{U} \\
        \overline{V}^{\intercal}
    \end{bmatrix}\begin{bmatrix}
        \overline{U} \\
        \overline{V}^{\intercal}
    \end{bmatrix}^{\intercal} \succeq 0.
\end{equation}

And finally, we show dual feasibility  by Shur's Complement.  Shur's complement gives us that because  $\overline{M}_1$ and $\overline{M}_4$ are positive semidefinite, then $\overline{M}$ is positive semidefinite if and only if 
$$\overline{M}_1-\overline{M}_2\overline{M}_4^{-1}\overline{M}_3=\frac{1}{2}I_{m\times m}-2\overline{M}_2\overline{M}_2^{\intercal}\succeq 0,$$
which is validated by our assumption 
$$||\mu \mathcal{B}^*(\mathcal{B}(\overline{X})-d) ||_2 \leq 1.$$

\end{proof}

\end{document}